\documentclass[12pt,reqno]{article}

\usepackage[usenames]{color}
\usepackage{amssymb}
\usepackage{graphicx}
\usepackage{amscd}

\usepackage{amsthm}
\newtheorem{theorem}{Theorem}

\newtheorem{proposition}[theorem]{Proposition}
\newtheorem{corollary}[theorem]{Corollary}

\theoremstyle{definition}

\newtheorem{example}[theorem]{Example}

\usepackage[colorlinks=true,
linkcolor=webgreen, filecolor=webbrown,
citecolor=webgreen]{hyperref}

\definecolor{webgreen}{rgb}{0,.5,0}
\definecolor{webbrown}{rgb}{.6,0,0}

\usepackage{color}

\usepackage{float}

\usepackage{graphics,amsmath,amssymb}
\usepackage{amsfonts}
\usepackage{latexsym}
\usepackage{epsf}

\begin{document}

\begin{center}
\vskip 1cm{\LARGE\bf Riordan arrays and the $LDU$ decomposition of symmetric Toeplitz plus Hankel matrices} \vskip 1cm \large
Paul Barry\\
School of Science\\
Aoife Hennessy\\
Department of Computing, Mathematics and Physics\\
Waterford Institute of Technology, Cork Road, Waterford, Ireland\\

\href{mailto:pbarry@wit.ie}{\tt pbarry@wit.ie}
\end{center}
\vskip .2 in

\begin{abstract} We examine a result of Basor and Ehrhardt concerning Hankel and Toeplitz plus Hankel matrices, within the context of the Riordan group of lower-triangular matrices. This allows us to determine the $LDU$ decomposition of certain symmetric Toeplitz plus Hankel matrices. We also determine the generating functions and Hankel transforms of associated sequences.
\end{abstract}


\section{Introduction}
In \cite{Basor} Basor and Ehrhardt studied the transformation
\begin{equation}
b_n=\sum_{k=0}^{n-1}\binom{n-1}{k}(a_{1-n+2k}+a_{2-n+2k}),\end{equation} defined for
sequences $\{a_n\}_{n=-\infty}^{\infty}$ in the context of relating the determinants of certain
Toeplitz plus Hankel matrices to the determinants of related Hankel matrices.

In this note, we shall study an equivalent transformation, which we will construct with the aid of Riordan arrays \cite{SGWW}. We call this the $\mathbb{B}$-transform. We shall then use our results to examine the $LDU$ decomposition of the resulting Toeplitz plus Hankel matrices.

In the next section, we shall detail the notations that will be used in this note, and give a basic introduction to the relevant theory of Riordan arrays. We shall follow this with a section which defines the $\mathbb{B}$-transform, studies some of its properties, and shows its equivalence the Basor and Ehrhardt transform. In particular, we derive an expression for the generating function of the image sequence, which for instance allows us to determine the Hankel transform of image sequence in many cases. A final section then looks at the $LDU$ decomposition of the related Toeplitz plus Hankel matrices, with examples involving Riordan arrays.
\section{Notation and basic Riordan array theory}
Although many of our results will be valid for sequences $a_n$ with values in $\mathbb{C}$, we shall in the sequel assume that
the sequences we deal with are integer sequences, $a_n \in \mathbb{Z}$.
For an integer sequence $a_n$, that is, an element of $\mathbb{Z}^\mathbb{N}$, the power series
$f(x)=\sum_{k=0}^{\infty}a_k x^k$ is called the \emph{ordinary generating function} or g.f. of the sequence.
$a_n$ is thus the coefficient of $x^n$ in this series. We denote this by
$a_n=[x^n]f(x)$. For instance, $F_n=[x^n]\frac{x}{1-x-x^2}$ is the $n$-th Fibonacci number, while
$C_n=[x^n]\frac{1-\sqrt{1-4x}}{2x}$ is the $n$-th Catalan number. We use the notation
$0^n=[x^n]1$ for the sequence $1,0,0,0,\ldots$ Thus $0^n=[n=0]=\delta_{n,0}=\binom{0}{n}$. Here,
we have used the Iverson bracket notation \cite{Concrete},
defined by $[\mathcal{P}]=1$ if the proposition $\mathcal{P}$
is true, and
$[\mathcal{P}]=0$ if $\mathcal{P}$ is false.

For a power series
$f(x)=\sum_{n=0}^{\infty}a_n x^n$ with $f(0)=0$ we define the reversion or compositional inverse of $f$ to be the
power series $\bar{f}(x)$ such that $f(\bar{f}(x))=x$. We sometimes write
$\bar{f}= \text{Rev}f$.

The Hankel transform \cite{Layman} of a sequence $a_n$ is the sequence $h_n=|a_{i+j}|_{i,j=0}^n$. If the sequence $a_n$ has a g.f. that has a continued fraction expansion of the form
$$\cfrac{a_0}{1-\alpha_0 x -
\cfrac{\beta_1 x^2}{1-\alpha_1 x-
\cfrac{\beta_2 x^2}{1-\alpha_2 x-
\cfrac{\beta_3 x^2}{1-\cdots}}}},$$ then the Hankel transform of $a_n$ is given by \cite{Kratt}
\begin{equation}h_n = a_0^{n+1} \beta_1^n\beta_2^{n-1}\cdots \beta_{n-1}^2\beta_n=a_0^{n+1}\prod_{k=1}^n
\beta_k^{n+1-k}.\end{equation} \noindent The $LDU$ decomposition of Hankel matrices has been studied in \cite{Triple, P_W}.

Some of the lower-triangular matrices that we shall meet will be coefficient arrays of families of orthogonal polynomials. General references for orthogonal polynomials include \cite{Chihara, Gautschi, Szego}.

$A^t$ will denote the transpose of the matrix $A$, and we will on occasion use $A\cdot B$ to denote the product of the matrices $AB$, where this makes reading the text easier. This also conforms with the use of ``$\cdot$'' for the product in the Riordan group (see below).

The \emph{Riordan group} \cite{SGWW, Spru}, is a set of
infinite lower-triangular integer matrices, where each matrix is
defined by a pair of generating functions
$g(x)=1+g_1x+g_2x^2+\cdots$ and $f(x)=f_1x+f_2x^2+\cdots$ where
$f_1\ne 0$ \cite{Spru}. We assume in addition that $f_1=1$ in what follows. The associated matrix is the matrix whose
$i$-th column is generated by $g(x)f(x)^i$ (the first column being
indexed by 0). The matrix corresponding to the pair $g, f$ is
denoted by $(g, f)$ or $\cal{R}$$(g,f)$. The group law is then given
by
\begin{displaymath} (g, f)\cdot(h, l)=(g, f)(h, l)=(g(h\circ f), l\circ
f).\end{displaymath} The identity for this law is $I=(1,x)$ and the
inverse of $(g, f)$ is $(g, f)^{-1}=(1/(g\circ \bar{f}), \bar{f})$
where $\bar{f}$ is the compositional inverse of $f$.

A Riordan array of the form $(g(x),x)$, where $g(x)$ is the
generating function of the sequence $a_n$, is called the
\emph{sequence array} of the sequence $a_n$. Its $(n,k)$-th term is
$a_{n-k}$. Such arrays are also called \emph{Appell} arrays as they form the elements of the
Appell subgroup.
\newline\newline If $\mathbf{M}$ is the matrix $(g,f)$, and
$\mathbf{a}=(a_0,a_1,\ldots)'$ is an integer sequence with ordinary
generating function $\cal{A}$ $(x)$, then the sequence
$\mathbf{M}\mathbf{a}$ has \cite{SGWW} ordinary generating function
$g(x)$$\cal{A}$$(f(x))$. This result is often called ``the Fundamental Theorem of Riordan arrays''. The (infinite) matrix $(g,f)$ can thus be considered to act on the ring of
integer sequences $\mathbb{Z}^\mathbb{N}$ by multiplication, where a sequence is regarded as a
(infinite) column vector. We can extend this action to the ring of power series
$\mathbb{Z}[[x]]$ by
$$(g,f):\cal{A}(\mathnormal{x}) \mapsto \mathnormal{(g,f)}\cdot
\cal{A}\mathnormal{(x)=g(x)}\cal{A}\mathnormal{(f(x))}.$$
\begin{example} The so-called \emph{binomial matrix} $\mathbf{B}$ is the element
$(\frac{1}{1-x},\frac{x}{1-x})$ of the Riordan group. It has general
element $\binom{n}{k}$, and hence as an array coincides with Pascal's triangle. More generally, $\mathbf{B}^m$ is the
element $(\frac{1}{1-m x},\frac{x}{1-mx})$ of the Riordan group,
with general term $\binom{n}{k}m^{n-k}$. It is easy to show that the
inverse $\mathbf{B}^{-m}$ of $\mathbf{B}^m$ is given by
$(\frac{1}{1+mx},\frac{x}{1+mx})$.
\end{example}
\noindent For a sequence $a_0, a_1, a_2, \ldots$ with g.f. $g(x)$, the ``aeration'' of the sequence is the sequence
$a_0, 0, a_1, 0, a_2, \ldots$ with interpolated zeros. Its g.f. is $g(x^2)$. We note that since $c(x)=\frac{1-\sqrt{1-4x}}{2x}$ has
the well-known continued fraction expansion
$$c(x)=\cfrac{1}{1-\cfrac{x}{1-\cfrac{x}{1-\ldots}}},$$ $c(x^2)$ has the continued fraction expansion
\begin{equation}\label{ACat} c(x^2)=\cfrac{1}{1-\cfrac{x^2}{1-\cfrac{x^2}{1-\ldots}}}.\end{equation}

The aeration of a (lower-triangular) matrix $\mathbf{M}$ with general term $m_{i,j}$ is the matrix whose general term is given by
$$m^r_{\frac{i+j}{2},\frac{i-j}{2}}\frac{1+(-1)^{i-j}}{2},$$ where
$m^r_{i,j}$ is the $i,j$-th element of the reversal of $\mathbf{M}$:
$$m^r_{i,j}=m_{i,i-j}.$$
In the case of a Riordan array (or indeed any lower triangular array), the row sums of the aeration are equal to the diagonal sums of
the reversal of the original matrix.
\begin{example}
The Riordan array $(c(x^2), xc(x^2))$ is the aeration of the Riordan array $$(c(x),xc(x))=(1-x,x(1-x))^{-1}.$$ Here
$$c(x)=\frac{1-\sqrt{1-4x}}{2x}$$ is the g.f. of the Catalan numbers.
Indeed, the reversal of $(c(x), xc(x))$ is the matrix with general element
$$[k\le n+1] \binom{n+k}{k}\frac{n-k+1}{n+1},$$ which begins

\begin{displaymath}\left(\begin{array}{ccccccc} 1 & 0 &
0
& 0 & 0 & 0 & \ldots \\1 & 1 & 0 & 0 & 0 & 0 & \ldots \\ 1 & 2
& 2 & 0 & 0 &
0 & \ldots \\ 1 & 3 & 5 & 5 & 0 & 0 & \ldots \\ 1 & 4 & 9
& 14 & 14 & 0 & \ldots \\1 & 5 & 14 & 28 & 42 & 42
&\ldots\\
\vdots &
\vdots & \vdots & \vdots & \vdots & \vdots &
\ddots\end{array}\right).\end{displaymath}
\noindent Then $(c(x^2),xc(x^2))$ has general element
$$\binom{n+1}{\frac{n-k}{2}}\frac{k+1}{n+1}\frac{1+(-1)^{n-k}}{2},$$ and begins
\begin{displaymath}\left(\begin{array}{ccccccc} 1 & 0 &
0
& 0 & 0 & 0 & \ldots \\0 & 1 & 0 & 0 & 0 & 0 & \ldots \\ 1 & 0
& 1 & 0 & 0 &
0 & \ldots \\ 0 & 2 & 0 & 1 & 0 & 0 & \ldots \\ 2 & 0 & 3
& 0 & 1 & 0 & \ldots \\0 & 5 & 0 & 4 & 0 & 1
&\ldots\\
\vdots &
\vdots & \vdots & \vdots & \vdots & \vdots &
\ddots\end{array}\right).\end{displaymath}
\noindent We have
$$(c(x^2),xc(x^2))=\left(\frac{1}{1+x^2},\frac{x}{1+x^2}\right)^{-1}.$$
\end{example}
\section{The $\mathbb{B}$-transform}
We let
$$L=\left(\frac{1-x}{1+x^2}, \frac{x}{1+x^2}\right)^{-1},$$ where we recall that the Riordan matrix
$(g,f)$ is the lower triangular matrix whose $k$-th column has generating function $g(x)f(x)^k$, for suitable
$g$ and $f$.
Then $L$ has $(n,k)$-th term
$$\binom{n}{\lfloor \frac{n-k}{2} \rfloor},$$ and $L^{-1}=\left(\frac{1-x}{1+x^2}, \frac{x}{1+x^2}\right)$ is the
coefficient array of the generalized Chebyshev polynomials defined by
$$P_n(x)=xP_{n-1}(x)-P_{n-2}(x),\quad P_0(x)=1,\quad P_1(x)=x-1.$$
The matrix $L$ begins
\begin{displaymath}L=\left(\begin{array}{ccccccc} 1 & 0 & 0
& 0 & 0 & 0 & \ldots \\1 & 1 & 0 & 0 & 0 & 0 & \ldots \\
2 & 1 & 1 & 0 & 0 & 0 & \ldots \\ 3 & 3 & 1 & 1
& 0 & 0 & \ldots \\ 6 & 4 & 4 & 1 & 1 & 0 & \ldots
\\10 & 10 & 5 & 5 & 1 & 1 &\ldots\\ \vdots & \vdots & \vdots & \vdots & \vdots
& \vdots & \ddots\end{array}\right).\end{displaymath}
We now define $\mathbb{B}$ to be the matrix
\begin{equation} \mathbb{B}=L \cdot (1+x,x)^t.\end{equation}
This matrix begins
\begin{displaymath}\mathbb{B}=\left(\begin{array}{ccccccc} 1 & 1 & 0
& 0 & 0 & 0 & \ldots \\1 & 2 & 1 & 0 & 0 & 0 & \ldots \\
2 & 3 & 2 & 1 & 0 & 0 & \ldots \\ 3 & 6 & 4 & 2
& 1 & 0 & \ldots \\ 6 & 10 & 8 & 5 & 2 & 1 & \ldots
\\10 & 20 & 15 & 10 & 6 & 2 &\ldots\\ \vdots & \vdots & \vdots & \vdots & \vdots
& \vdots & \ddots\end{array}\right).\end{displaymath}
Since the matrix $(1+x,x)^t$ is given by
\begin{eqnarray*}(1+x,x)^t&=&
\left(\begin{array}{ccccccc} 1 & 1 & 0
& 0 & 0 & 0 & \ldots \\0 & 1 & 1 & 0 & 0 & 0 & \ldots \\
0 & 0 & 1 & 1 & 0 & 0 & \ldots \\ 0 & 0 & 0 & 1
& 1 & 0 & \ldots \\ 0 & 0 & 0 & 0 & 1 & 1 & \ldots
\\0 & 0 & 0 & 0 & 0 & 1 &\ldots\\ \vdots & \vdots & \vdots & \vdots & \vdots
& \vdots & \ddots\end{array}\right)\\
&=&\left(\begin{array}{ccccccc} 1 & 0 & 0
& 0 & 0 & 0 & \ldots \\0 & 1 & 0 & 0 & 0 & 0 & \ldots \\
0 & 0 & 1 & 0 & 0 & 0 & \ldots \\ 0 & 0 & 0 & 1
& 0 & 0 & \ldots \\ 0 & 0 & 0 & 0 & 1 & 0 & \ldots
\\0 & 0 & 0 & 0 & 0 & 1 &\ldots\\ \vdots & \vdots & \vdots & \vdots & \vdots
& \vdots & \ddots\end{array}\right)+
\left(\begin{array}{ccccccc} 0 & 1 & 0
& 0 & 0 & 0 & \ldots \\0 & 0 & 1 & 0 & 0 & 0 & \ldots \\
0 & 0 & 0 & 1 & 0 & 0 & \ldots \\ 0 & 0 & 0 & 0
& 1 & 0 & \ldots \\ 0 & 0 & 0 & 0 & 0 & 1 & \ldots
\\0 & 0 & 0 & 0 & 0 & 0 &\ldots\\ \vdots & \vdots & \vdots & \vdots & \vdots
& \vdots & \ddots\end{array}\right),\end{eqnarray*} we see that the $(n,k)$-th element of $\mathbb{B}$ is given by
$$b_{n,k}=\binom{n}{\lfloor \frac{n-k}{2} \rfloor}+\binom{n}{\lfloor \frac{n-k+1}{2} \rfloor}-\binom{n}{\lfloor \frac{n}{2} \rfloor}\cdot 0^k.$$
Now let $\{a_n\}_{n \ge 0}$ be a sequence. We define the $\mathbb{B}$ \emph{transform} of $a_n$ to be the sequence $\{b_n\}_{n\ge 0}$ given by
\begin{equation} b_n=\sum_{k=0}^{n+1} b_{n,k} a_k\end{equation} where
$\mathbb{B}=(b_{n,k})_{n,k \ge 0}$.
\begin{example} The $\mathbb{B}$-transform of the Fibonacci numbers is given by
$$b_n=\sum_{k=0}^{n+1} b_{n,k} F_k=\sum_{k=0}^{n+1}(\binom{n}{\lfloor \frac{n-k}{2} \rfloor}+\binom{n}{\lfloor \frac{n-k+1}{2} \rfloor}-\binom{n}{\lfloor \frac{n}{2} \rfloor}\cdot 0^k)F_k.$$ This sequence starts
$$1, 3, 7, 17, 39, 91, 207, 475, 1075, 2445, 5515,\ldots.$$
It has the interesting property that its Hankel transform is $(-2)^n$.
\end{example}
\begin{proposition} We have
\begin{equation} \mathbb{B}=\left(\frac{1}{1+x^2},\frac{x}{1+x^2}\right)^{-1}\cdot \mathbb{T},\end{equation} where
$\mathbb{T}$ is the matrix
\begin{equation} \mathbb{T}=\left(\frac{1}{1-x},x\right)\cdot (1+x,x)^t.\end{equation}
\end{proposition}
\begin{proof}
We have
$$\mathbb{B}=L\cdot (1+x,x)^t=L \cdot \left(\frac{1}{1-x},x\right)^{-1}\cdot \left(\frac{1}{1-x},x\right)\cdot (1+x,x)^t.$$
Now
\begin{eqnarray*}L\cdot \left(\frac{1}{1-x},x\right)^{-1}&=&
\left(\frac{1-x}{1+x^2},\frac{x}{1+x^2}\right)^{-1}\cdot \left(\frac{1}{1-x},x\right)^{-1}\\
&=&\left(\left(\frac{1}{1-x},x\right)\cdot \left(\frac{1-x}{1+x^2},\frac{x}{1+x^2}\right)\right)^{-1}\\
&=&\left(\frac{1}{1+x^2},\frac{x}{1+x^2}\right)^{-1}.\end{eqnarray*}
\end{proof}
\noindent We recall that the  matrix $\left(\frac{1}{1+x^2},\frac{x}{1+x^2}\right)^{-1}=(c(x^2),xc(x^2))$, where
$c(x)=\frac{1-\sqrt{1-4x}}{2x}$ is the g.f. of the Catalan numbers, has
general element
$$\binom{n+1}{\frac{n-k}{2}}\frac{k+1}{n+1}\frac{1+(-1)^{n-k}}{2}.$$
In addition,
$\left(\frac{1}{1-x},x\right)\cdot (1+x,x)^t$ is the matrix $\mathbb{T}$ given by
\begin{displaymath}\mathbb{T}=\left(\begin{array}{ccccccc} 1 & 1 & 0
& 0 & 0 & 0 & \ldots \\1 & 2 & 1 & 0 & 0 & 0 & \ldots \\
1 & 2 & 2 & 1 & 0 & 0 & \ldots \\ 1 & 2 & 2 & 2
& 1 & 0 & \ldots \\ 1 & 2 & 2 & 2 & 2 & 1 & \ldots
\\1 & 2 & 2 & 2 & 2 & 2 &\ldots\\ \vdots & \vdots & \vdots & \vdots & \vdots
& \vdots & \ddots\end{array}\right).\end{displaymath}
Now note that
\begin{displaymath}\mathbb{T}=\left(\begin{array}{ccccccc} 1 & 0 & 0
& 0 & 0 & 0 & \ldots \\1 & 1 & 0 & 0 & 0 & 0 & \ldots \\
1 & 1 & 1 & 0 & 0 & 0 & \ldots \\ 1 & 1 & 1 & 1
& 0 & 0 & \ldots \\ 1 & 1 & 1 & 1 & 1 & 0 & \ldots
\\1 & 1 & 1 & 1 & 1 & 1 &\ldots\\ \vdots & \vdots & \vdots & \vdots & \vdots
& \vdots & \ddots\end{array}\right)+\left(\begin{array}{ccccccc} 0 & 1 & 0
& 0 & 0 & 0 & \ldots \\0 & 1 & 1 & 0 & 0 & 0 & \ldots \\
0 & 1 & 1 & 1 & 0 & 0 & \ldots \\ 0 & 1 & 1 & 1
& 1 & 0 & \ldots \\ 0 & 1 & 1 & 1 & 1 & 1 & \ldots
\\0 & 1 & 1 & 1 & 1 & 1 &\ldots\\ \vdots & \vdots & \vdots & \vdots & \vdots
& \vdots & \ddots\end{array}\right), \end{displaymath} and hence the action of
$\mathbb{T}$ on a sequence $a_n$ is to return the sequence with $n$-th term equal to
$$\sum_{k=0}^n a_k + \sum_{k=1}^{n+1} a_k = 2 \sum_{k=0}^n a_k +a_{n+1}-a_0.$$
Thus we have
\begin{proposition}
We have
$$b_n=\sum_{k=0}^n \binom{n+1}{\frac{n-k}{2}}\frac{k+1}{n+1}\frac{1+(-1)^{n-k}}{2}(\sum_{j=0}^k a_j + \sum_{j=1}^{k+1} a_j),$$ or equivalently,
$$b_n=\sum_{k=0}^n \binom{n+1}{\frac{n-k}{2}}\frac{k+1}{n+1}\frac{1+(-1)^{n-k}}{2}(2\sum_{j=0}^k a_j + a_{k+1}-a_0).$$
\end{proposition} In the following, we will be interested in determining the g.f. of the image of $a_n$. We have the following result.
\begin{proposition} Let $f(x)$ be the g.f. of $a_n$. Then $b_n=\sum_{k=0}^{n+1} b_{n,k}a_k$ has g.f. given by
\begin{equation}\left(\frac{1}{1+x^2},\frac{x}{1+x^2}\right)^{-1}\cdot \left(\frac{(1+x)f(x)-a_0}{x(1-x)}\right).\end{equation}
Equivalently, the g.f. of $b_n$ is given by
\begin{equation}\label{Equiv}\left(\frac{1-x}{1+x^2},\frac{x}{1+x^2}\right)^{-1}\cdot \left(\frac{(1+x)f(x)-a_0}{x}\right).\end{equation}
\end{proposition}
\begin{proof}
The result follows from the fact that the generating function of $\sum_{k=0}^n a_k+\sum_{k=1}^{n+1} a_k$ is given by
$$\frac{1}{1-x} f(x)+\frac{1}{1-x}\left(\frac{f(x)-a_0}{x}\right)=\frac{(1+x)f(x)-a_0}{x(1-x)}.$$
\end{proof}
\begin{corollary} The g.f. of $b_n=\sum_{k=0}^n b_{n,k} a_k$ is given by
$$\left(\frac{(1+xc(x^2))f(xc(x^2))-a_0}{x(1-xc(x^2))}\right).$$
\end{corollary}
\begin{proof}
This follows from the fundamental theorem of Riordan arrays since
$$\left(\frac{1}{1+x^2},\frac{x}{1+x^2}\right)^{-1}=(c(x^2),xc(x^2)).$$
\end{proof}
\begin{example} The g.f. of the $\mathbb{B}$-transform of the Fibonacci numbers $F_n$ is given by
$$\frac{1+xc(x^2)}{x(1-xc(x^2))}\frac{xc(x^2)}{1-xc(x^2)-x^2c(x^2)^2}.$$
This follows since
$$F_n=[x^n]\frac{x}{1-x-x^2}$$ and $F_0=0$. This g.f. may be simplified to
$$\frac{1-4x^2+x\sqrt{1-4x^2}}{1-2x-5x^2+10x^3}=\frac{1-4x^2+x\sqrt{1-4x^2}}{(1-2x)(1-5x^2)}.$$
By solving the equation
$$u=\cfrac{1}{1-3x+\cfrac{x^2}{1+x-x^2c(x^2)}},$$ we see that this g.f. may be expressed (using  Eq. (\ref{ACat}))
as the continued fraction
$$\cfrac{1}{1-3x+
\cfrac{2x^2}{1+x-
\cfrac{x^2}{1-
\cfrac{x^2}{1-
\cfrac{x^2}{1-\ldots}}}}},$$
which shows that the the Hankel transform of the $\mathbb{B}$-transform of the Fibonacci numbers is $(-2)^n$.
\end{example}

We have defined the $\mathbb{B}$ matrix using the Riordan array $\left(\frac{1}{1+x^2},\frac{x}{1+x^2}\right)$. This matrix is associated with the Chebyshev polynomials of the second kind $U_n(x)$ (it is the coefficient array of $U_n(x/2)$). The matrix $\left(\frac{1-x^2}{1+x^2},\frac{x}{1+x^2}\right)$ is related to the Chebyshev polynomials of the first kind $T_n$. We have
\begin{proposition} We have
\begin{equation}\label{Cheb1} \mathbb{B}=\left(\frac{1-x^2}{1+x^2},\frac{x}{1+x^2}\right)^{-1}\cdot (1+x,x)\cdot (1+x,x)^t.\end{equation}
\end{proposition}
\begin{proof} We have
\begin{eqnarray*} \mathbb{B}&=&L\cdot (1+x,x)^t=L\cdot (1+x,x)^{-1}\cdot (1+x,x)\cdot(1+x,x)^t\\
&=&L\cdot \left(\frac{1}{1+x},x\right)\cdot (1+x,x)\cdot(1+x,x)^t\\
&=& \left(\frac{1-x}{1+x^2},\frac{x}{1+x^2}\right)\cdot\left(\frac{1}{1+x},x\right)\cdot (1+x,x)\cdot(1+x,x)^t\\
&=&\left(\frac{1-x^2}{1+x^2},\frac{x}{1+x^2}\right)^{-1}\cdot (1+x,x)\cdot (1+x,x)^t.\end{eqnarray*}
\end{proof}
\noindent We can decompose $(1+x,x)\cdot (1+x,x)^t$ as the sum of two matrices:
\begin{displaymath}\left(\begin{array}{ccccccc} 1 & 0 & 0
& 0 & 0 & 0 & \ldots \\1 & 1 & 0 & 0 & 0 & 0 & \ldots \\
0 & 1 & 1 & 0 & 0 & 0 & \ldots \\ 0 & 0 & 1 & 1
& 0 & 0 & \ldots \\ 0 & 0 & 0 & 1 & 1 & 0 & \ldots
\\0 & 0 & 0 & 0 & 1 & 1 &\ldots\\ \vdots & \vdots & \vdots & \vdots & \vdots
& \vdots & \ddots\end{array}\right)+\left(\begin{array}{ccccccc} 0 & 1 & 0
& 0 & 0 & 0 & \ldots \\0 & 1 & 1 & 0 & 0 & 0 & \ldots \\
0 & 0 & 1 & 1 & 0 & 0 & \ldots \\ 0 & 0 & 0 & 1
& 1 & 0 & \ldots \\ 0 & 0 & 0 & 0 & 1 & 1 & \ldots
\\0 & 0 & 0 & 0 & 0 & 1 &\ldots\\ \vdots & \vdots & \vdots & \vdots & \vdots
& \vdots & \ddots\end{array}\right), \end{displaymath} which is the sum of
$(1+x,x)$ and a shifted version of $(1+x,x)$.
To obtain $\mathbb{B}$ we multiply by $\left(\frac{1-x^2}{1+x^2},\frac{x}{1+x^2}\right)^{-1}$.
This gives us, once again
\begin{displaymath}\mathbb{B}=\left(\begin{array}{ccccccc} 1 & 0 & 0
& 0 & 0 & 0 & \ldots \\1 & 1 & 0 & 0 & 0 & 0 & \ldots \\
2 & 1 & 1 & 0 & 0 & 0 & \ldots \\ 3 & 3 & 1 & 1
& 0 & 0 & \ldots \\ 6 & 4 & 4 & 1 & 1 & 0 & \ldots
\\10 & 10 & 5 & 5 & 1 & 1 &\ldots\\ \vdots & \vdots & \vdots & \vdots & \vdots
& \vdots & \ddots\end{array}\right)+\left(\begin{array}{ccccccc} 0 & 1 & 0
& 0 & 0 & 0 & \ldots \\0 & 1 & 1 & 0 & 0 & 0 & \ldots \\
0 & 2 & 1 & 1 & 0 & 0 & \ldots \\ 0 & 3 & 3 & 1
& 1 & 0 & \ldots \\ 0 & 6 & 4 & 4 & 1 & 1 & \ldots
\\0 & 10 & 10 & 5 & 5 & 1 &\ldots\\ \vdots & \vdots & \vdots & \vdots & \vdots
& \vdots & \ddots\end{array}\right), \end{displaymath}

\noindent where the first member of the sum is the Riordan array
$$\left(\frac{1-x^2}{1+x^2},\frac{x}{1+x^2}\right)^{-1}\cdot (1+x,x)=\left(\frac{1-x^2}{(1+x)(1+x^2)},\frac{x}{1+x^2}\right)^{-1}=L.$$
\begin{theorem}
Let $b_n=\sum_{k=0}^{n+1} b_{n,k} a_k$ where $b_{n,k}$ is the $(n,k)$-th element of $\mathbb{B}$. Then
$$b_n=\sum_{k=0}^n \binom{n}{k}(a_{n-2k}+a_{n-2k+1}),$$ where we have extended $a_n$ to negative $n$ by setting
$a_{-n}=a_n$.
\end{theorem}
\begin{proof}

We have seen that $\mathbb{B}$ has general term
$$\binom{n}{\lfloor \frac{n-k}{2} \rfloor}+\binom{n}{\lfloor \frac{n-k+1}{2} \rfloor}-0^k\cdot \binom{n}{\lfloor \frac{n}{2} \rfloor}.$$
Thus the $\mathbb{B}$ transform of $a_n$ is given by
$$\sum_{k=0}^{n+1}\left(\binom{n}{\lfloor \frac{n-k}{2} \rfloor}+\binom{n}{\lfloor \frac{n-k+1}{2} \rfloor}-0^k\cdot \binom{n}{\lfloor \frac{n}{2} \rfloor}\right)a_k$$ which can also be written as
$$\sum_{k=0}^{n+1}\left(\binom{n}{\lfloor \frac{n-k}{2} \rfloor}(1-0^k)+\binom{n}{\lfloor \frac{n-k+1}{2} \rfloor} \right)a_k$$
since $0^k\cdot \binom{n}{\lfloor \frac{n}{2} \rfloor}=0^k\cdot \binom{n}{\lfloor \frac{n-k}{2} \rfloor}$.
We can also write this as
\begin{equation}\label{b_n}b_n=\sum_{k=0}^{n+1}\left(\binom{n}{\lfloor \frac{k-1}{2} \rfloor}(1-0^{n-k+1})+\binom{n}{\lfloor \frac{k}{2} \rfloor} \right)a_{n-k+1}.\end{equation}
Now note that
\begin{eqnarray*}\sum_{k=0}^n \binom{n}{k}(a_{n-2k}+a_{n-2k+1})&=&\sum_{k=0}^{\lfloor \frac{n}{2} \rfloor}\binom{n}{k}(a_{n-2k}+a_{n-2k+1})+\sum_{k=\lfloor \frac{n}{2} \rfloor+1}^n\binom{n}{k}(a_{n-2k}+a_{n-2k+1})\\
&=&\sum_{k=0}^{\lfloor \frac{n}{2} \rfloor}\binom{n}{k}(a_{n-2k}+a_{n-2k+1})+\sum_{k=\lfloor \frac{n}{2} \rfloor+1}^n\binom{n}{n-k}(a_{2k-n}+a_{2k-n-1}).\end{eqnarray*}
 By gathering similar terms in the above expression, and considering the cases of $n$ even ($n \backslash
 2=0$) and $n$ odd, we arrive at
\begin{equation}\sum_{k=0}^n \binom{n}{k}(a_{n-2k}+a_{n-2k+1})=\sum_{k=0}^{\lfloor \frac{n-1}{2} \rfloor}\binom{n}{k}(a_{n-2k+1}+2 a_{n-2k}+a_{n-2k-1})+[n \backslash 2=0]\binom{n}{\lfloor \frac{n}{2} \rfloor}(a_0+a_1).\end{equation}
By considering the separate sums for $k$ even and $k$ odd in Eq. (\ref{b_n}), extending to negative $n$ and gathering terms we find that also
$$b_n=\sum_{k=0}^{\lfloor \frac{n-1}{2} \rfloor}\binom{n}{k}(a_{n-2k+1}+2 a_{n-2k}+a_{n-2k-1})+[n \backslash 2=0]\binom{n}{\lfloor \frac{n}{2} \rfloor}(a_0+a_1).$$

\end{proof}
\noindent Thus we have the following equivalent expressions:
\begin{eqnarray*}
b_n&=&\sum_{k=0}^{n+1} b_{n,k} a_k\\
&=&\sum_{k=0}^n \binom{n}{k}(a_{n-2k}+a_{n-2k+1})\\
&=&\sum_{k=0}^{\lfloor \frac{n-1}{2} \rfloor}\binom{n}{k}(a_{n-2k+1}+2 a_{n-2k}+a_{n-2k-1})+[n \backslash 2=0]\binom{n}{\lfloor \frac{n}{2} \rfloor}(a_0+a_1)\\
&=&\sum_{k=0}^n \binom{n+1}{\frac{n-k}{2}}\frac{k+1}{n+1}\frac{1+(-1)^{n-k}}{2}(\sum_{j=0}^k a_j+ \sum_{j=1}^{k+1} a_j).\end{eqnarray*}
\section{Symmetric Toeplitz plus Hankel matrices}
We now recall result Proposition 2.1 from \cite{Basor}, which we state in the language used above.
\begin{proposition} \cite[Proposition 2.1]{Basor}. Let $(a_n)_{n=-\infty}^n$ be a sequence with $a_n=a_{-n}$ and let
$$b_n=\sum_{k=0}^n \binom{n}{k}(a_{n-2k}+a_{n-2k+1}).$$ Also let $H=(b_{i+j})_{i,j \ge 0}$ be the Hankel matrix of $(b_n)_{n\ge 0}$ and
$A=(a_{i-j}+a_{i+j+1})_{i,j\ge 0}$ be the Toeplitz plus Hankel matrix associated to $(a_n)_{n=-\infty}^n$. Finally let $L$ be the matrix with $(n,k)-$th term $\binom{n}{\lfloor \frac{n-k}{2} \rfloor}$.
Then
\begin{equation} H = L\cdot A \cdot L^t.\end{equation}
\end{proposition}
\noindent An immediate consequence of this is that
$$ A = L^{-1} H (L^t)^{-1} = L^{-1} H (L^{-1})^t.$$

\noindent If now $H$ has an $LDU$ decomposition $H=\mathcal{L}\cdot D\cdot \mathcal{L}^t$ then we obtain an $LDU$ decomposition
for the symmetric Toeplitz plus Hankel matrix $A$:
$$ A = L^{-1} \cdot \mathcal{L}\cdot D\cdot \mathcal{L}^t \cdot (L^{-1})^t,$$
or
\begin{equation} A = (L^{-1}\mathcal{L})\cdot D \cdot (L^{-1}\mathcal{L})^t.\end{equation}
\begin{example} We continue our example with the Fibonacci numbers. Thus let
$$b_n = \sum_{k=0}^{n+1} b_{n,k} F_k = [x^n]\frac{1-4x^2+x\sqrt{1-4x^2}}{(1-2x)(1-5x^2)}.$$ For this sequence, we have the
following $LDU$ decomposition of $H=(b_{i+j})_{i,j \ge 0}$.
\begin{displaymath}
\left(\begin{array}{ccccccc} 1 & 3 & 7
& 17 & 39 & 91 & \ldots \\3 & 7 & 17 & 39 & 91 & 207 & \ldots \\
7 & 17 & 39 & 91 & 207 & 475 & \ldots \\ 17 & 39 & 91 & 207
& 475 & 1075 & \ldots \\ 39 & 91 & 207 & 475 & 1075 & 2445 & \ldots
\\91 & 207 & 475 & 1075 & 2445 & 5515 &\ldots\\ \vdots & \vdots & \vdots & \vdots & \vdots
& \vdots & \ddots\end{array}\right)=\end{displaymath}
\begin{displaymath}\left(\begin{array}{ccccccc} 1 & 0 & 0
& 0 & 0 & 0 & \ldots \\3 & 1 & 0 & 0 & 0 & 0 & \ldots \\
7 & 2 & 1 & 0 & 0 & 0 & \ldots \\ 12 & 6 & 2 & 1
& 0 & 0 & \ldots \\ 39 & 13 & 7 & 2 & 1 & 0 & \ldots
\\91 & 33 & 15 & 8 & 2 & 1 &\ldots\\ \vdots & \vdots & \vdots & \vdots & \vdots
& \vdots & \ddots\end{array}\right)\left(\begin{array}{ccccccc} 1 & 0 & 0
& 0 & 0 & 0 & \ldots \\0 & -2 & 0 & 0 & 0 & 0 & \ldots \\
0 & 0 & -2 & 0 & 0 & 0 & \ldots \\ 0 & 0 & 0 & -2
& 0 & 0 & \ldots \\ 0 & 0 & 0 & 0 & -2 & 0 & \ldots
\\0 & 0 & 0 & 0 & 0 & -2 &\ldots\\ \vdots & \vdots & \vdots & \vdots & \vdots
& \vdots & \ddots\end{array}\right)\left(\begin{array}{ccccccc} 1 & 0 & 0
& 0 & 0 & 0 & \ldots \\3 & 1 & 0 & 0 & 0 & 0 & \ldots \\
7 & 2 & 1 & 0 & 0 & 0 & \ldots \\ 12 & 6 & 2 & 1
& 0 & 0 & \ldots \\ 39 & 13 & 7 & 2 & 1 & 0 & \ldots
\\91 & 33 & 15 & 8 & 2 & 1 &\ldots\\ \vdots & \vdots & \vdots & \vdots & \vdots
& \vdots & \ddots\end{array}\right)^t\end{displaymath}
Here, the first matrix $\mathcal{L}$ of the product is the inverse of the coefficient array of the orthogonal polynomials for which the sequence $b_n$ is the moment sequence.  These polynomials are specified by
$$P_n(x)=xP_{n-1}(x)-P_{n-2}(x), \quad P_0(x)=1,P_1(x)=x-3,P_2(x)=x^2-2x-1.$$
We have
$$L^{-1}\mathcal{L}=\left(\begin{array}{ccccccc} 1 & 0 & 0
& 0 & 0 & 0 & \ldots \\2 & 1 & 0 & 0 & 0 & 0 & \ldots \\
3 & 1 & 1 & 0 & 0 & 0 & \ldots \\ 5 & 2 & 1 & 1
& 0 & 0 & \ldots \\ 8 & 3 & 2 & 1 & 1 & 0 & \ldots
\\13 & 5 & 3 & 2 & 1 & 1 &\ldots\\ \vdots & \vdots & \vdots & \vdots & \vdots
& \vdots & \ddots\end{array}\right),$$ and thus

\begin{displaymath}
A=\left(\begin{array}{ccccccc} 1 & 2 & 3
& 5 & 8 & 13 & \ldots \\2 & 2 & 4 & 6 & 10 & 15 & \ldots \\
3 & 4 & 5 & 9 & 14 & 23 & \ldots \\ 5 & 6 & 9 & 13
& 22 & 35 & \ldots \\ 8 & 10 & 14 & 22 & 34 & 56 & \ldots
\\13 & 16 & 23 & 35 & 56 & 89 &\ldots\\ \vdots & \vdots & \vdots & \vdots & \vdots
& \vdots & \ddots\end{array}\right)=\end{displaymath}
$$\left(\begin{array}{ccccccc} 1 & 0 & 0
& 0 & 0 & 0 & \ldots \\2 & 1 & 0 & 0 & 0 & 0 & \ldots \\
3 & 1 & 1 & 0 & 0 & 0 & \ldots \\ 5 & 2 & 1 & 1
& 0 & 0 & \ldots \\ 8 & 3 & 2 & 1 & 1 & 0 & \ldots
\\13 & 5 & 3 & 2 & 1 & 1 &\ldots\\ \vdots & \vdots & \vdots & \vdots & \vdots
& \vdots & \ddots\end{array}\right) \left(\begin{array}{ccccccc} 1 & 0 & 0
& 0 & 0 & 0 & \ldots \\0 & -2 & 0 & 0 & 0 & 0 & \ldots \\
0 & 0 & -2 & 0 & 0 & 0 & \ldots \\ 0 & 0 & 0 & -2
& 0 & 0 & \ldots \\ 0 & 0 & 0 & 0 & -2 & 0 & \ldots
\\0 & 0 & 0 & 0 & 0 & -2 &\ldots\\ \vdots & \vdots & \vdots & \vdots & \vdots
& \vdots & \ddots\end{array}\right)
\left(\begin{array}{ccccccc} 1 & 0 & 0
& 0 & 0 & 0 & \ldots \\2 & 1 & 0 & 0 & 0 & 0 & \ldots \\
3 & 1 & 1 & 0 & 0 & 0 & \ldots \\ 5 & 2 & 1 & 1
& 0 & 0 & \ldots \\ 8 & 3 & 2 & 1 & 1 & 0 & \ldots
\\13 & 5 & 3 & 2 & 1 & 1 &\ldots\\ \vdots & \vdots & \vdots & \vdots & \vdots
& \vdots & \ddots\end{array}\right)^t.$$
We note that the matrix $L^{-1}\mathcal{L}$ in this case is ``almost'' a Riordan array, in that it is
the Fibonacci ``sequence-array'' $\left(\frac{1}{1-x-x^2},x\right)$ with general term
$[k \le n]F_{n-k+1}$, shifted once with a first column of $F_{n+2}$ pre-pended.
\end{example}
\begin{example} We take the example of the Jacobsthal numbers
$$J_n=\frac{2^n}{3}-\frac{(-1)^n}{3}=[x^n]\frac{x}{1-x-2x^2}.$$ We note that this is the the element
corresponding to $r=2$ of the family of sequences with $n$-th term given by
$$[x^n]\frac{x}{1-x-rx^2}=\sum_{k=0}^{\lfloor \frac{n-1}{2} \rfloor} \binom{n-k-1}{k}r^k, $$ where the
Fibonacci numbers correspond to $r=1$.
Thus we let
$$b_n = \sum_{k=0}^{n+1} b_{n,k} J_k.$$
Then the g.f. for $b_n$ is given by
$$(c(x), xc(x^2))\cdot  \left(\frac{(1+x)\left(\frac{x}{1-x-2x^2}\right)-0}{x(1-x)}\right)=\frac{\sqrt{1-4x^2}+3(1-2x)}{2(2-9x+10x^2)}.$$
This is equivalent to the expansion
$$\cfrac{1}{1-3x+
\cfrac{x^2}{1-
\cfrac{x^2}{1-
\cfrac{x^2}{1-\cdots}}}},$$ from which we deduce that the Hankel transform of the
$\mathbb{B}$-transform of the Jacobsthal numbers is $(-1)^n$.
Using Eq. (\ref{Equiv}), we can also write the g.f. of $b_n$ as
$$\left(\frac{1-x}{1+x^2},\frac{x}{1+x^2}\right)^{-1}\cdot \left(\frac{(1+x)\frac{x}{(1+x)(1-2x)}}{x}\right)=\left(\frac{1-x}{1+x^2},\frac{x}{1+x^2}\right)^{-1}\cdot \frac{1}{1-2x},$$ and
hence we have
$$\sum_{k=0}^{n+1} b_{n,k}J_k=\sum_{k=0}^n \binom{n}{\lfloor \frac{n-k}{2} \rfloor}2^k.$$ The Hankel matrix $H$ for $b_n$ has $LDU$ decomposition $\mathcal{L} D \mathcal{L}^t$ as follows:
\begin{displaymath}
H=\left(\begin{array}{ccccccc} 1 & 3 & 8
& 21 & 54 & 138 & \ldots \\3 & 8 & 21 & 54 & 138 & 350 & \ldots \\
8 & 21 & 54 & 138 & 350 & 885 & \ldots \\ 21 &54 & 138 & 350
& 885 & 2230 & \ldots \\ 54 & 138 & 350 & 885 & 2230 & 5610 & \ldots
\\138 & 350 & 885 & 2230 & 5610 & 14088 &\ldots\\ \vdots & \vdots & \vdots & \vdots & \vdots
& \vdots & \ddots\end{array}\right)=\end{displaymath}
$$\left(\begin{array}{ccccccc} 1 & 0 & 0
& 0 & 0 & 0 & \ldots \\3 & 1 & 0 & 0 & 0 & 0 & \ldots \\
8 & 3 & 1 & 0 & 0 & 0 & \ldots \\ 21 & 9 & 3 & 1
& 0 & 0 & \ldots \\ 54 & 24 & 10 & 3 & 1 & 0 & \ldots
\\138 & 64 & 27 & 3 & 1 & 1 &\ldots\\ \vdots & \vdots & \vdots & \vdots & \vdots
& \vdots & \ddots\end{array}\right) \left(\begin{array}{ccccccc} 1 & 0 & 0
& 0 & 0 & 0 & \ldots \\0 & -1 & 0 & 0 & 0 & 0 & \ldots \\
0 & 0 & -1 & 0 & 0 & 0 & \ldots \\ 0 & 0 & 0 & -1
& 0 & 0 & \ldots \\ 0 & 0 & 0 & 0 & -1 & 0 & \ldots
\\0 & 0 & 0 & 0 & 0 & -1 &\ldots\\ \vdots & \vdots & \vdots & \vdots & \vdots
& \vdots & \ddots\end{array}\right)
\left(\begin{array}{ccccccc} 1 & 0 & 0
& 0 & 0 & 0 & \ldots \\3 & 1 & 0 & 0 & 0 & 0 & \ldots \\
8 & 3 & 1 & 0 & 0 & 0 & \ldots \\ 21 & 9 & 3 & 1
& 0 & 0 & \ldots \\ 54 & 24 & 10 & 3 & 1 & 0 & \ldots
\\138 & 64 & 27 & 3 & 1 & 1 &\ldots\\ \vdots & \vdots & \vdots & \vdots & \vdots
& \vdots & \ddots\end{array}\right)^t.$$
In this case, the matrix $\mathcal{L}$ is a Riordan array, equal to
$$\mathcal{L}=\left(\frac{\sqrt{1-4x^2}+3(1-2x)}{2(2-9x+10x^2)},xc(x^2)\right)=\left(\frac{1-3x+2x^2}{1+x^2},\frac{x}{1+x^2}\right)^{-1}.$$
Here, $\left(\frac{1-3x+2x^2}{1+x^2},\frac{x}{1+x^2}\right)$ is the coefficient array of the family of orthogonal polynomials given by $$P_n(x)=xP_{n-1}-P_{n-2}(x), \quad P_0(x)=1,P_1(x)=x-3,P_2(x)=x^2-3x+1.$$ The $\mathbb{B}$-transform of the Jacobsthal numbers $J_n$ is thus the moment sequence for this family of orthogonal polynomials. Finally, we have
\begin{eqnarray*}L^{-1}\mathcal{L}&=&\left(\frac{1-x}{1+x^2},\frac{x}{1+x^2}\right)\cdot \left(\frac{\sqrt{1-4x^2}+3(1-2x)}{2(2-9x+10x^2)},xc(x^2)\right)\\
&=&\left(\frac{1-x}{1+x^2},\frac{x}{1+x^2}\right)\cdot\left(\frac{1-3x+2x^2}{1+x^2},\frac{x}{1+x^2}\right)\\
&=&\left(\frac{1}{1-2x},x\right).\end{eqnarray*} Thus the Toeplitz plus Hankel matrix $A$ associated to the Jacobsthal numbers $J_n$ has $LDU$ decomposition
$$A=\left(\begin{array}{ccccccc} 1 & 2 & 4
& 8 & 16 & 32 & \ldots \\2 & 3 & 6 & 12 & 24 & 48 & \ldots \\
4 & 6 & 11 & 22 & 44 & 88 & \ldots \\ 8 & 12 & 22 & 43
& 86 & 172 & \ldots \\ 16 & 24 & 44 & 86 & 171 & 342 & \ldots
\\32 & 48 & 88 & 172 & 342 & 683 &\ldots\\ \vdots & \vdots & \vdots & \vdots & \vdots
& \vdots & \ddots\end{array}\right)=$$
$$\left(\begin{array}{ccccccc} 1 & 0 & 0
& 0 & 0 & 0 & \ldots \\2 & 1 & 0 & 0 & 0 & 0 & \ldots \\
4 & 2 & 1 & 0 & 0 & 0 & \ldots \\ 8 & 4 & 2 & 1
& 0 & 0 & \ldots \\ 16 & 8 & 4 & 2 & 1 & 0 & \ldots
\\32 & 16 & 8 & 4 & 2 & 1 &\ldots\\ \vdots & \vdots & \vdots & \vdots & \vdots
& \vdots & \ddots\end{array}\right) \left(\begin{array}{ccccccc} 1 & 0 & 0
& 0 & 0 & 0 & \ldots \\0 & -1 & 0 & 0 & 0 & 0 & \ldots \\
0 & 0 & -1 & 0 & 0 & 0 & \ldots \\ 0 & 0 & 0 & -1
& 0 & 0 & \ldots \\ 0 & 0 & 0 & 0 & -1 & 0 & \ldots
\\0 & 0 & 0 & 0 & 0 & -1 &\ldots\\ \vdots & \vdots & \vdots & \vdots & \vdots
& \vdots & \ddots\end{array}\right)
\left(\begin{array}{ccccccc} 1 & 0 & 0
& 0 & 0 & 0 & \ldots \\2 & 1 & 0 & 0 & 0 & 0 & \ldots \\
4 & 2 & 1 & 0 & 0 & 0 & \ldots \\ 8 & 4 & 2 & 1
& 0 & 0 & \ldots \\ 16 & 8 & 4 & 2 & 1 & 0 & \ldots
\\32 & 16 & 8 & 4 & 2 & 1 &\ldots\\ \vdots & \vdots & \vdots & \vdots & \vdots
& \vdots & \ddots\end{array}\right)^t.$$

\end{example}

\end{document}